\documentclass[11pt]{article}
\usepackage{amsmath, amsthm}
\usepackage{latexsym}
\usepackage{graphicx, psfrag}
\makeatletter
\newfont{\footsc}{cmcsc10 at 8truept}
\newfont{\footbf}{cmbx10 at 8truept}
\newfont{\footrm}{cmr10 at 10truept}
\makeatother
\newtheorem{theorem}{\bf Theorem}

\newtheorem{lemma}{\bf Lemma}

\begin{document}
\title{Positivity of the Rational Function of Gillis, Reznick and Zeilberger}

\author{Yaming Yu\\
\small Department of Statistics\\[-0.8ex]
\small University of California\\[-0.8ex] 
\small Irvine, CA 92697, USA\\[-0.8ex]
\small \texttt{yamingy@uci.edu}}

\date{}
\maketitle

\begin{abstract}
We prove that the power series expansion of the rational function of Gillis, Reznick and Zeilberger (1983) has only nonnegative coefficients.
\end{abstract}

\begin{section}{Introduction}
Deciding whether a multivariate rational function has all positive coefficients in its Taylor expansion around the origin is a classical topic dating back to  Szeg\"{o} (1933), with multiple open conjectures.  Recently much progress has been made, including Scott and Sokal's (2014) insightful solution of the Lewy-Askey problem which shows that the expansion of
$(1-\sum_{i=1}^4 t_i +(2/3)\sum_{1\leq i<j\leq 4} t_it_j)^{-1}$ has only positive coefficients.  There is also a long-standing conjecture of Gillis, Reznick and Zeilberger (1983) which states that for $r\geq 4$ the expansion of $(1-\sum_{i=1}^r t_i +r!\prod_{i=1}^r t_i)^{-1}$ has only nonnegative coefficients (henceforth, the GRZ conjecture).  Using symbolic computation, Kauers (2007) has shown that the GRZ conjecture holds for $r=4,5,6$.  This is extended by Pillwein (2019) to $r=7, \ldots, 17$.  We take a less computer-intensive approach and are able to prove the GRZ conjecture in full.  A new ingredient in our method is analyzing the locations of roots of some real-rooted polynomials derived from the GRZ rational function.  Such analyses can conceivably be performed on other rational functions and resolve similar positivity questions. 

In Section~2 we present our proof.  For related results, see, for example, Straub and Zudilin (2015). 
\end{section}

\begin{section}{Proof of Main Result}
\begin{lemma}
\label{lem1}
For $r=2,3,\ldots$ we have the following formal power series
\begin{equation}
\label{uni}
\frac{1}{1-at + bt^r} = \sum_{n=0}^\infty \sum_{l=0}^{r-1} t^{nr+l} \left[a^l \sum_{k=0}^n (-1)^k {r(n-k)+k+l \choose k} a^{r(n-k)} b^k\right].
\end{equation}
\end{lemma}
\begin{proof}
Multiply and compare coefficients.
\end{proof}

\begin{lemma}
\label{lem2}
Assume $a=r,\ b\leq (r-1)^{r-1}$ and $r=2,3\ldots$.  Then (\ref{uni}) has all positive coefficients as a univariate power series in $t$.
\end{lemma}
\begin{proof}
If $b\leq 0$ then $(1-rt+bt^r)^{-1} = \sum_{n=0}^\infty (rt-bt^r)^n$ which evidently has all positive coefficients. 
Suppose $0<b\leq (r-1)^{r-1}$.  Then the minimum of $g(t)\equiv 1-rt+bt^r$ on $t\in (0, \infty)$ is achieved when $g'(t) = brt^{r-1} -r =0$, that is, when $t=  t^* \equiv b^{-1/(r-1)}$.  Moreover, $g(t_*) = 1-(r-1)b^{-1/(r-1)} \leq 0$ by assumption.  Hence $g(t)=0$ has at least one positive solution, say $t_0$, and we may write 
 $$g(t) = g(t)- g(t_0) = (t_0 - t)\left(r-b\sum_{k=0}^{r-1} t_0^k t^{r-1-k}\right).$$
Since $t_0>0,\, b>0$ and $r-bt_0^{r-1} = 1/t_0>0$, upon inverting $t_0-t$ and the other term separately we see that $g^{-1}(t)$ has all positive coefficients.
\end{proof}

\begin{lemma}
\label{lem3}
For $r=2, 3, \ldots, l=0,1\ldots, r-1,$ and $n=1,2,\ldots,$ all roots of the polynomial in $s$
$$h(s) \equiv  \sum_{k=0}^n (-1)^k {r(n-k)+k+l \choose k} s^{n-k}$$
are real and lie within the interval $(0, r^r/(r-1)^{r-1})\subset(0, er)$ with $e=2.71828\cdots$.
\end{lemma}
\begin{proof}
Yu (2009) has proved that if the sequence of binomial coefficients located on a ray of Pascal's triangle is a finite sequence, then it is a Polya frequency sequence, as conjectured by Su and Wang (2008).  That $h(s)$ has only real roots is a special case of this.  Since the coefficients of $h(s)$ alternate in sign, only positive roots are possible.  Let $\alpha$ be one of the roots and suppose $\alpha \geq r^r/(r-1)^{r-1}$.  Then we may choose $a=r$ and $b= r^r/\alpha \leq (r-1)^{r-1}$ in (\ref{uni}) and conclude by Lemma~\ref{lem2} that all coefficients of the power series expansion of $(1-rt+bt^r)^{-1}$ are positive, i.e.,  $0<\sum_{k=0}^n (-1)^k {r(n-k)+k+l \choose k} a^{r(n-k)} b^{k} =b^n h(\alpha)$ which contradicts $h(\alpha) =0$.
 \end{proof}
 
 \begin{lemma}
 \label{lem4}
 Let $e_1=\sum_{i=1}^r t_i,\ e_r = \prod_{i=1}^r t_i$ and let $c=r^rr!/(r-1)^{r-1}$.  Let $k, l$ be nonnegative integers such that $k+l\geq 2$.  Then for $r\geq 8$, the $r$-variate polynomials in $t_1, \ldots, t_r$ defined by $(e_1^r - ce_r)^k e_1^{rl}$ have positive coefficients only.
 \end{lemma}
 \begin{proof}
We only need to show that $(e_1^r - ce_r)^ke_1^{rl}$ has positive coefficients with $(k, l)=(1,1), (2, 0)$ or $(3, 0)$.  The cases with larger $k$ and $l$ can then be obtained by multiplying suitable copies of these three.  Consider $(k, l)=(1,1)$ which corresponds to $e_1^{2r} - ce_r e_1^r$.  For $\beta_i\geq 1$ such that $\sum_{i=1}^r \beta_i = 2r$, the coefficient of $t_1^{\beta_1}\cdots t_r^{\beta_r}$ in $e_1^{2r} - ce_r e_1^r$ is given by ${2r \choose \beta_1,\ldots,\beta_r} - c {r \choose \beta_1 -1, \ldots, \beta_r-1}$, which is positive because $(2r)!/(r!\beta_1\ldots\beta_r)\geq (2r)!/(r!2^r) > c$ for $r\geq 7$.  We omit similar proofs in the other two cases, which entail $r\geq 8$.
   \end{proof}

\begin{lemma}
\label{lem5}
In the setting of Lemma~\ref{lem4} let $n\geq 2$ and $\alpha_1,\ldots, \alpha_n \in [0, c]$.  Then for $r\geq 8$ the $r$-variate polynomial in $t_1, \ldots, t_r$ defined by $\prod_{i=1}^n (e_1^r - \alpha_i e_r)$ has positive coefficients only.
\end{lemma}

\begin{proof}
Let $m = \#\{i:\ 0<\alpha_i<c\}$. We use induction on $m$.  If $m=0$ then all $\alpha_i$ are either $0$ or $c$, and the claim follows from Lemma~\ref{lem4}.  If $m\geq 1$ then suppose $0<\alpha_1<c$.  Letting $\lambda= \alpha_1/c$ we have
$$\prod_{i=1}^n (e_1^r - \alpha_i e_r) = \left(\lambda(e_1^r - ce_r) + (1-\lambda)e_r^r\right)\prod_{i=2}^n (e_1^r - \alpha_i e_r).$$
The claim follows from the induction hypothesis. 
\end{proof}   

 \begin{theorem}
 \label{thm}
If $r\geq 8$ then the expansion of $(1-\sum_{i=1}^r t_i + r! \prod_{i=1}^r t_i)^{-1}$ as a power series in $t_1,\ldots, t_r$ around the origin has only nonnegative coefficients.
 \end{theorem}  
\begin{proof}
Define $a=e_1=\sum_{i=1}^r t_i$ and $b= r!e_r = r!\prod_{i=1}^r t_i$.  By Lemma~\ref{lem1}, we only need to show that for $n\geq 0$ and $l=0,\ldots, r-1$ the polynomial in $t_1,\ldots, t_r$ defined by 
$$p(t_1,\ldots, t_r)\equiv a^l \sum_{k=0}^n (-1)^k {r(n-k)+k+l \choose k} a^{r(n-k)} b^{k}$$ 
has only nonnegative coefficients.  This is trivial for $n=0$.  For $n=1$, we need to show that, for $l=0, \ldots, r-1$, all coefficients of $e_1^l(e_1^r - (l+1) r! e_r)$ are nonnegative, which can be verified by similar arguments as in the proof of Lemma~\ref{lem4}.  Suppose $n\geq 2$.  By Lemma~\ref{lem3}, there exist $\alpha_1,\ldots, \alpha_r\in (0, r^r/(r-1)^{r-1})$ such that $p(t_1,\ldots, t_r) = a^l \prod_{i=1}^n (a^r - \alpha_i b)$ which, by Lemma~\ref{lem5}, has only positive coefficients as a polynomial in $(t_1, \ldots, t_r)$, assuming $r\geq 8$.
\end{proof}

{\bf Remark.} For smaller values of $r$ we may refine Lemma~\ref{lem4} by proving that $(e_1^r - ce_r)^k e_1^{rl}$ has only positive coefficients for all $k+l\geq n_0$ where $n_0$ is a positive integer depending on $r$.  We can then extend Theorem~\ref{thm} to $r=7, 6, 5, 4$.  We omit the details as Kauers (2007) and Pillwein (2019) have already proved the claimed positivity in such cases, although our approach here circumvents Proposition 3 of Gillis, Reznick and Zeilberger (1983) and may yield a shorter overall proof of the GRZ conjecture.

\end{section}

\end{document}